\documentclass[conference]{IEEEtran}%
\usepackage{amssymb}
\usepackage{amsmath}
\usepackage{amsfonts}
\usepackage{hyperref}%
\usepackage{graphicx}
\providecommand{\U}[1]{\protect\rule{.1in}{.1in}}
\newtheorem{theorem}{Theorem}

\newtheorem{corollary}[theorem]{Corollary}

\newtheorem{proposition}[theorem]{Proposition}

\begin{document}

\title{Mathematical modelling and optimal control of anthracnose}%

%TCIMACRO{\TeXButton{Author Information}{\author{
%\authorblockN{David FOTSA}
%\authorblockA{
%ENSAI\\
%The University of Ngaoundere\\
%Email: mjdavidfotsa@gmail.com
%}
%\and\authorblockN{Elvis HOUPA, David BEKOLLE}
%\authorblockA{
%Faculty of Science\\
%The University of Ngaoundere\\
%Email: e-houpa@yahoo.com, bekolle@yahoo.fr
%}
%\and\authorblockN{Christopher THRON}
%\authorblockA{ \\
%Texas A\&M University, Central Texas\\
%Email:thron@ct.tamus.edu
%}
%\and\authorblockN{Michel NDOUMBE}
%\authorblockA{
%IRAD, Cameroon\\
%Email: michel.ndoumbe@yahoo.com}
%}}}%
%BeginExpansion
\author{
\authorblockN{David FOTSA}
\authorblockA{
ENSAI\\
The University of Ngaoundere\\
Email: mjdavidfotsa@gmail.com
}
\and\authorblockN{Elvis HOUPA, David BEKOLLE}
\authorblockA{
Faculty of Science\\
The University of Ngaoundere\\
Email: e-houpa@yahoo.com, bekolle@yahoo.fr
}
\and\authorblockN{Christopher THRON}
\authorblockA{ \\
Texas A\&M University, Central Texas\\
Email:thron@ct.tamus.edu
}
\and\authorblockN{Michel NDOUMBE}
\authorblockA{
IRAD, Cameroon\\
Email: michel.ndoumbe@yahoo.com}
}%
%EndExpansion
%

%TCIMACRO{\TeXButton{Make Title}{\maketitle}}%
%BeginExpansion
\maketitle
%EndExpansion
%

%TCIMACRO{\TeXButton{Begin abstract}{\begin{abstract}}}%
%BeginExpansion
\begin{abstract}%
%EndExpansion
In this paper we propose two nonlinear models for the control of anthracnose
disease. The first is an ordinary differential equation (ODE) model which
represents the within-host evolution of the disease. The second includes
spatial diffusion of the disease in a bounded domain. We demonstrate the
well-posedness of those models by verifying the existence of solutions for
given initial conditions and positive invariance of the positive cone. By
considering a quadratic cost functional and applying a maximum principle, we
construct a feedback optimal control for the ODE model which is evaluated
through numerical simulations with the scientific software
Scilab\textregistered. For the diffusion model we establish under some
conditions the existence of a unique optimal control with respect to a
generalized version of the cost functional mentioned above. We also provide a
characterization for this optimal control.



%TCIMACRO{\TeXButton{End abstract}{\end{abstract}}}%
%BeginExpansion
\end{abstract}%
%EndExpansion

\section{Introduction}

\qquad Anthracnose is a phytopathology which attacks a wide range of
commercial crops, including almond, mango, banana, blueberry, cherry, citrus,
coffee, hevea and strawberry. The disease has been identified in such diverse
areas as Ceylon (1923), Guadeloupe (1925), Sumatra (1929), Indochina (1930),
Costa Rica (1931), Malaysia (1932), Java (1933), Madagascar (1934), Cameroon
(1934), Colombia (1940), Salvador (1944), Brazil (1946), Nyassaland (1949),
New Caledonia (1954), and Arabia (1956) \cite{boisson}. Anthracnose can affect
various parts of the plant, including leaves, fruits, twigs and roots.
Possible symptoms include defoliation, fruit rot, fruit fall and crown root
rot, which can occur before or after harvest depending on both pathogen and
host \cite{boisson,wharton}.
%TCIMACRO{\FRAME{ftbphFU}{3.5993in}{2.3298in}{0pt}{\Qcb{Symptoms of Coffee
%Berry Disease (CBD) \cite{bieysse}}}{\Qlb{CBDSymptoms}}{Figure}%
%{\special{ language "Scientific Word";  type "GRAPHIC";  display "USEDEF";
%valid_file "F";  width 3.5993in;  height 2.3298in;  depth 0pt;
%original-width 6.7153in;  original-height 4.3336in;  cropleft "0";
%croptop "1";  cropright "1";  cropbottom "0";
%filename 'cbdsymptoms.jpg';file-properties "XNPEU";}} }%
%BeginExpansion
\begin{figure}
[ptbh]
\begin{center}
\includegraphics[
natheight=4.333600in,
natwidth=6.715300in,
height=2.3298in,
width=3.5993in
]%
{cbdsymptoms.jpg}%
\caption{Symptoms of Coffee Berry Disease (CBD) \cite{bieysse}}%
\label{CBDSymptoms}%
\end{center}
\end{figure}
%EndExpansion

The Anthracnose pathogen belongs to the \textit{Colletotrichum} species
(\textit{acutatum}, \textit{capsici,} \textit{gloeosporioides, kahawae,
lindemuthianum, musae, ...}). \textit{Colletotrichum} is an ascomycete fungus.
It can reproduce either asexually or sexually, but sexual reproduction is rare
in nature \cite{wharton}. Favourable growth conditions occur particularly in
tropical zones. Rainfall, wetness and altitude are all conducive to
sporulation and conidia spreading \cite{mouen09,muller70}. Sources of inoculum
are thought to be leaves, buds and mummified fruits.

\subsection{Anthracnose pathosystem}

\qquad The process of infection by \textit{Colletotrichum} species can usually
be divided into at least seven steps, depending on various factors including
growth conditions, host tissues and involved species. Conidia deposited on the
host attach themselve on its surface. The conidia germinate after 12--48
hours, and appressoria are produced \cite{boisson,jeffries}. Severals studies
on infection chronology show that appressoria production can occur between
3--48 hours following germination under favourable conditions of wetness and
temperature \cite{mouen09,wharton}. The pathogen then penetrates the plant
epidermis, invades plant tissues, produces acevuli and finally sporulates. The
penetration of plant epidermis is enabled by a narrow penetration peg that
emerges from the appressorium base \cite{chen}. In some marginal cases
penetration occurs through plant tissues' stomata or wounds. Once the cuticle
is crossed, two infection strategies can be distiguished: intracellular
hemibiotrophy and subcuticular intramural necrotrophy, as shown in
Figure~\ref{HostInvasion}. Invasion of the host is led through formation of
hyphae which narrow as the infection progresses. \textit{Colletotrichum}
produce enzymes that degrade carbohydrates, dissolve cell walls, and hydrolyze
cuticle. Some of those enzymes are polyglacturonases, pectin lyases and
proteases. Some hosts may employ various biochemical strategies to counter the
pathogen. For example, the peel of unripe avocados has been found in vitro to
contain a preformed antifungal diene (cis,
cis-1-acetoxy-2-hydroxy-4-oxo-heneicosa-12, 15-diene) that inhibits the growth
of \textit{Colletotrichum gloeosporioides} when present above a certain
concentration \cite{wharton}.%

%TCIMACRO{\FRAME{ftbphFU}{3.2197in}{1.3889in}{0pt}{\Qcb{Infection strategies.
%(A)=Apressorium - (C)=Conidium - (Cu)=Cuticle - (E)=Epidermal - (ILS)=Internal
%Light Spot - (M)=Mesophyl cell - (N)=Necrotrophic - (PH)=Primary Hyphae -
%(PP)=Penetration Peg - (ScH)=Subcuticular and Intramural Hyphae -
%(SH)=Secondary Hyphae \cite{wharton}}}{\Qlb{HostInvasion}}%
%{evolutioncolletotrichumii.jpg}{\special{ language "Scientific Word";
%type "GRAPHIC";  maintain-aspect-ratio TRUE;  display "USEDEF";
%valid_file "F";  width 3.2197in;  height 1.3889in;  depth 0pt;
%original-width 10.0837in;  original-height 4.3336in;  cropleft "0";
%croptop "1";  cropright "1";  cropbottom "0";
%filename 'EvolutionColletotrichumII.JPG';file-properties "XNPEU";}} }%
%BeginExpansion
\begin{figure}
[ptbh]
\begin{center}
\includegraphics[
natheight=4.333600in,
natwidth=10.083700in,
height=1.3889in,
width=3.2197in
]%
{EvolutionColletotrichumII.jpg}%
\caption{Infection strategies. (A)=Apressorium - (C)=Conidium - (Cu)=Cuticle -
(E)=Epidermal - (ILS)=Internal Light Spot - (M)=Mesophyl cell -
(N)=Necrotrophic - (PH)=Primary Hyphae - (PP)=Penetration Peg -
(ScH)=Subcuticular and Intramural Hyphae - (SH)=Secondary Hyphae
\cite{wharton}}%
\label{HostInvasion}%
\end{center}
\end{figure}
%EndExpansion

\subsection{Models in the literature}

\qquad Most previous mathematical studies on Colletotrichum-host pathosystem
have focused on forecasting disease onset based on environmental factors
affecting host sensitivity. DANNEBERGER et al. in \cite{danneberger} have
developed a forecasting model for the annual bluegrass anthracnose severity
index, using weather elements such as temperature and wetness. Their model is
a quadratic regression
\[
ASI=a_{0}+a_{0,1}W+a_{1,0}T+a_{1,1}T\times W+a_{0,2}T^{2}+a_{2,0}W^{2}%
\]
where $ASI$ is the anthracnose severity index, $T$ is the daily average
temperature and $W$ is the average number of hours of leaves' wetness per day.
DODD et al. in \cite{dodd} have studied the relationship between temperature
$(T)$, relative humidity $(H)$, incubation period $(t)$ and the percentage
$(p)$ of conidia of Colletotrichum gloeosporioides producing pigmented
appressoria on one month old mangoes. They used the following logistic model:
\[
\ln\left(  p/\left(  1-p\right)  \right)  =a_{0}+a_{0,1}H+a_{1,0}%
T+a_{0,2}H^{2}+a_{2,0}T^{2}+b\ln\left(  t\right)
\]
DUTHIE in \cite{duthie} examines the parasite's response $(R)$ to the combined
effects of temperature $(T)$ and wetness duration $(W)$. That response could
be the rate of germination, infection efficiency, latent period, lesion
density, disease incidence or disease severity. Several models are discussed,
the two principal being
\[
R\left(  T,W\right)  =f\left(  T\right)  \left[  1-\exp\left(  -\left[
b\left(  W-c\right)  \right]  ^{d}\right)  \right]
\]
and
\[
R\left(  T,W\right)  =a\left[  1-\exp\left(  -\left[  f\left(  T\right)
\left(  W-c\right)  \right]  ^{d}\right)  \right]  ,
\]
where
\[
f\left(  T\right)  =\frac{e\left(  1+h\right)  h^{\frac{h}{1+h}}}{\left(
1+\exp\left(  g\left[  T-f\right]  \right)  \right)  }\exp\left(
\frac{g\left[  T-f\right]  }{1+h}\right)
\]
and
\[
a>0,\text{ }b>0\text{, }W\geq c\geq0\text{, }d>0\text{, }e>0\text{, }%
f\geq0\text{, }g>0\text{, }h>0.
\]
MOUEN et al. attempt in \cite{mouen07} to develop a spatio-temporal model to
analyse infection behaviour with respect to the time, and identify potential
foci for disease inoculum. Logistic regression and kriging tools are used
used. In addition to these references, there are several other statistical
models in literature
\cite{duthie,mouen07,mouen09,mouen072,mouen03,mouen08,wharton}.

\subsection{Controlling anthracnose}

\qquad There are many approaches to controlling anthracnose diseases. The
genetic approach involves selection or synthesis of more resistant cultivars
\cite{bella,bieysse,boisson,ganesh,silva}. Several studies have demonstrated
the impact of cultivational practices on disease dynamics
\cite{mouen072,mouen03,mouen08,wharton}. Other tactics may be used to reduce
predisposition and enhance resistance, such as pruning old infected twigs,
removing mummified fruits, and shading \cite{boisson}. Biological control uses
microorganisms or biological substrates which interact with pathogen or induce
resistance in the host \cite{durand}. Finally there is chemical control, which
requires the periodic application of antifungal compounds
\cite{boisson,muller67,muller71}. This seems to be the most reliable method,
though relatively expensive. The best control policy should schedule different
approaches to optimize quality, quantity and cost of production. Note that
inadequate application of treatments could induce resistance in the pathogen
\cite{ramos}.

\subsection{Organization of the paper}

\qquad The remainder of this paper is organized as follows. In section
\ref{WithinHostModel} we propose and study a within-host model of anthracnose.
We present that model and give parameters meaning in subsection
\ref{WithinHostModel1}. Throughout subsection \ref{WithinHostModel2} we
establish the well-posedness of the within-host model both in mathematical and
epidemiological senses. The optimal control of the model is surveyed in
subsection \ref{WithinHostModel3} and numerical simulations are performed in
the last subsection \ref{WithinHostModel4}. We make a similar study on a
spatial version of the model includind a diffusion term in section
\ref{sec:diffusion}. That last model is presented in subsection
\ref{DiffusionModel1}. Studies on its well-posedness and its optimal control
are made respectively in subsections \ref{DiffusionModel2} and
\ref{DiffusionModel3}. Finally, in section \ref{Discussion} we discuss our
modelling and some realistic generalizations which could be added to the model.

\section{A within-host model\label{WithinHostModel}}

\subsection{Specification of the within-host model\label{WithinHostModel1}}

The detrimental effects of \textit{Colletotrichum} infection on fruit growth
are closely related to its life cycle. It is mathematically convenient to
express these effects in terms of the effective inhibition rate (denoted by
$\theta$), which is a continuous function of time. The effective inhibition
rate is defined such that the maximum attainable fruit volume is reduced by a
factor $1-\theta$ if current infection conditions are maintained. In addition
to $\theta$, the other time-dependent variables in the model are host fruit
total volume and infected volume, denoted by $v$ and $v_{r}$ respectively. We
have on the set $S=%
%TCIMACRO{\U{211d} }%
%BeginExpansion
\mathbb{R}
%EndExpansion
_{+}\setminus\left\{  1\right\}  \times%
%TCIMACRO{\U{211d} }%
%BeginExpansion
\mathbb{R}
%EndExpansion
_{+}^{\ast}\times%
%TCIMACRO{\U{211d} }%
%BeginExpansion
\mathbb{R}
%EndExpansion
_{+}$ the following equations for the time-evolution of the variables
($\theta,v,v_{r})$:%

\begin{equation}
\left\{
\begin{array}
[c]{l}%
d\theta/dt=\alpha\left(  t,\theta\right)  \left(  1-\theta/\left(
1-\theta_{_{1}}u\left(  t\right)  \right)  \right) \\
dv/dt=\beta\left(  t,\theta\right)  \left(  1-v\theta_{_{2}}/\left(  \left(
1-\theta\right)  \eta\left(  t\right)  v_{\max}\right)  \right) \\
dv_{r}/dt=\gamma\left(  t,\theta\right)  \left(  1-v_{r}/v\right)
\end{array}
\right. \label{ModelIntraHote}%
\end{equation}
The parameters in $\left(  \ref{ModelIntraHote}\right)  $ have the following
practical interpretations:

\begin{itemize}
\item $\alpha,\beta,\gamma$ characterize the effects of environmental and
climatic conditions on the rate of change of inhibition rate, fruit volume,
and infected fruit volume respectively. These are all positive functions of
the time $t$ and inhibition rate $\theta$.

\item $\gamma$ is an increasing function with respect to $\theta$ and
satisfies $\gamma\left(  t,0\right)  =0$, $\forall t\geq0$.

\item $u$ is a measurable control parameter which takes values in the set
$\left[  0,1\right]  $.

\item $1-\theta_{_{1}}\in\left[  0,1\right]  $ is the inhibition rate
corresponding to epidermis penetration. Once the epidermis has been
penetrated, the inhibition rate cannot fall below this value, even under
maximum control effort. In the absence of control effort $(u(t)=0)$, the
inhibition rate increases towards 1.

\item $\eta$ is a function of time that characterizes the effects of
environmental and climatic conditions on the maximum fruit volume. Its range
is the interval $\left]  0,\theta_{_{2}}\right]  $.

\item $v_{\max}$ represents the maximum size of the fruit.

\item $1-\theta_{_{2}}\in\left[  0,1\right]  $ is the value of inhibition rate
$\theta$ that corresponds to a limiting fruit volume of $\eta v_{\max} $.
According to the second equation in (\ref{ModelIntraHote}), the limiting
volume size is $\eta v_{\max}\left(  1-\theta\right)  /\theta_{_{2}}\leq
v_{\max}$. When the volume is less than this value, it increases (but never
passes the limiting value); while if the volume exceeds this value, then it
decreases. This limiting value for $v$ is less than $\eta v_{\max}$ when
$\theta>1-\theta_{_{2}}$ (note $\eta\leq\theta_{2}\leq1$).
\end{itemize}

Note that equations $\left(  \ref{ModelIntraHote}\right)  $ are constructed so
that $v \le v_{\max}$ and $v_{r} \le v$ as long as initial conditions satisfy
these inequailities.

With the definitions
\[
A\equiv%
\begin{bmatrix}
\frac{-\alpha\left(  t,\theta\right)  }{\left(  1-\theta_{_{1}}u\left(
t\right)  \right)  } & 0 & 0\\
0 & -\frac{\theta_{_{2}}\beta\left(  t,\theta\right)  }{\left(  \left(
1-\theta\right)  \eta\left(  t\right)  v_{\max}\right)  } & 0\\
0 & 0 & -\frac{\gamma\left(  t,\theta\right)  }{v}%
\end{bmatrix}
,
\]%
\[
B\equiv%
\begin{bmatrix}
\alpha\left(  t,\theta\right)  & \beta\left(  t,\theta\right)  & \gamma\left(
t,\theta\right)
\end{bmatrix}
^{T},
\]
and
\[
X\equiv%
\begin{bmatrix}
\theta & v & v_{r}%
\end{bmatrix}
^{T},
\]
then model $\left(  \ref{ModelIntraHote}\right)  $ can be reformulated as
\begin{equation}
dX/dt=F\left(  t,X\right)  ,\label{EqGeneral}%
\end{equation}
where
\begin{equation}
F\left(  t,X\right)  \equiv A\left(  t,X,u\right)  X+B\left(  t,X\right)  .
\end{equation}

As indicated above, model $\left(  \ref{ModelIntraHote}\right)  $ is an
exclusively within-host model, and as such does not include the effects of
spreading from host to host. (In Section~\ref{sec:diffusion} we propose a
diffusion model for between-host spreading.) Such a model has several
practical advantages. In practice, monitoring of the spreading of the fungi
population is difficult. Furthermore, conidia sources and spreading mechanisms
are not well-understood, although the literature generally points to mummified
fruits, leaves and bark as sources of inoculum. Instead of controlling the
host-to-host transmission, an alternative control method is to slow down the
within-host fungi evolution process. Such an approach enables the use of
statistical methods, since large samples of infected hosts may easily be
obtained \cite{jeffries}.

\subsection{Well-posedness of the within-host model\label{WithinHostModel2}}

In the following discussion, we demonstrate that model $\left(
\ref{ModelIntraHote}\right)  $ is well-posed both mathematically and
epidemiologically, under the following standard technical assumptions:

$\left(  H1\right)  $ The control parameter $u$ is measurable.

$\left(  H2\right)  $ The function $F$ is continuous with respect to the
variable $X$.

$\left(  H3\right)  $ For every compact subset $K$ $\subset S$, there is an
integrable map $M_{K}:%
%TCIMACRO{\U{211d} }%
%BeginExpansion
\mathbb{R}
%EndExpansion
_{+}\rightarrow%
%TCIMACRO{\U{211d} }%
%BeginExpansion
\mathbb{R}
%EndExpansion
_{+}$ such that for every $X$ in $K$ and $t$ in $%
%TCIMACRO{\U{211d} }%
%BeginExpansion
\mathbb{R}
%EndExpansion
_{+}$, $\left\Vert F\left(  t,X\right)  \right\Vert _{S}\leq M_{K}\left(
t\right)  $.

Existence of a solution is guaranteed by the following proposition, which
follows from a simple application of the Carath\'{e}odory theorem.

\begin{proposition}
For every initial condition $\left(  t_{0},X_{0}\right)  $ in $%
%TCIMACRO{\U{211d} }%
%BeginExpansion
\mathbb{R}
%EndExpansion
_{+}\times S$ there is a function $X\left(  t_{0},X_{0},t\right)  $ which is
absolutely continuous and satisfies $\left(  \ref{EqGeneral}\right)  $ for
almost any time $t\in%
%TCIMACRO{\U{211d} }%
%BeginExpansion
\mathbb{R}
%EndExpansion
_{+}$.
\end{proposition}

Uniqueness and smoothness of the solution may be established using the
Cauchy-Lipschitz Theorem, based on properties ($H2$) and ($H3$) of the
function $F$.

Next we will etablish positive invariance of the set $S$, and the positive
invariance of a bounded subset $BS$. These results are needed to show
consistency of the biological interpretation of the solution, as explained
below. With the definitions%

\[
A_{1}\equiv%
\begin{bmatrix}
-\frac{\alpha\left(  t,\theta\right)  }{\left(  1-\theta_{_{1}}u\left(
t\right)  \right)  } & 0 & 0\\
0 & -\frac{\theta_{_{2}}}{\left(  \left(  1-\theta\right)  \eta\left(
t\right)  v_{\max}\right)  } & 0\\
0 & 0 & -\frac{1}{v}%
\end{bmatrix}
,
\]%
\[
A_{2}\equiv%
\begin{bmatrix}
\alpha\left(  t,\theta\right)  & 0 & 0\\
0 & \beta\left(  t,\theta\right)  & 0\\
0 & 0 & \gamma\left(  t,\theta\right)
\end{bmatrix}
,
\]%
\[
B_{1}\equiv%
\begin{bmatrix}
1 & 1 & 1
\end{bmatrix}
^{T},
\]
and $X$ as defined above, then model $\left(  \ref{ModelIntraHote}\right)  $
can be reformulated as
\begin{equation}
dX/dt=A_{2}\left(  A_{1}X+B_{1}\right) .\label{EqGeneral2}%
\end{equation}

\begin{theorem}
The set $S$ is positively invariant for the system $\left(  \ref{EqGeneral2}%
\right)  $.
\end{theorem}

\begin{proof}
A solution to $\left(  \ref{EqGeneral2}\right)  $\ satisfies for every time
$t\geq0$,
\begin{align*}
X\left(  t\right)   &  =\exp\left[  \int\nolimits_{0}^{t}A_{2}\left(
s\right)  \cdot A_{1}\left(  s\right)  ds\right]  X\left(  0\right) \\
&  +\int\nolimits_{0}^{t}\exp\left[  \int\nolimits_{s}^{t}A_{2}\left(
\xi\right)  A_{1}\left(  \xi\right)  d\xi\right]  A_{2}\left(  s\right)
B_{1}ds
\end{align*}
Since $-A_{2}\left(  s\right)  A_{1}\left(  s\right)  $ is a $M-$matrix for
every time $s\geq0$, $\exp\left[  \int\nolimits_{s}^{t}A_{2}\left(
\xi\right)  A_{1}\left(  \xi\right)  d\xi\right]  $ is a positive matrix.
Moreover, since $B_{1}$ is nonnegtive, one can conclude that $X$ remain
nonnegative when $X\left(  0\right)  $ is taken nonnegative.
\end{proof}

\begin{theorem}
Let $BS$ be the subset of $S$ defined such as
\[
BS=\left\{  \left(  \theta,v,v_{r}\right)  \in%
%TCIMACRO{\U{211d} }%
%BeginExpansion
\mathbb{R}
%EndExpansion
^{3};0\leq\theta<1,0<v\leq v_{\max},0\leq v_{r}\leq v\right\}
\]
Then $BS$ is positively invariant for system $\left(  \ref{EqGeneral2}\right)
$.
\end{theorem}

\begin{proof}
We will show that at each point of the boundary of $BS$, the system $\left(
\ref{EqGeneral2}\right)  $ returns into $BS$. We prove this by showing that
the scalar product of the system time derivative with the normal vector $n$ at
each boundary point is nonpositive. It has been already shown that positive
orthant is positively invariant. Let
\[
F_{1}\equiv\left\{  \left(  \theta,v,v_{r}\right)  \in BS;\theta=1\right\}
\]%
\[
F_{2}\equiv\left\{  \left(  \theta,v,v_{r}\right)  \in BS;v=v_{\max}\right\}
\]%
\[
F_{3}\equiv\left\{  \left(  \theta,v,v_{r}\right)  \in BS;v_{r}=v\right\}
\]
For all points on $F_{1}$, $n$ can be choosen as $\left(  1,0,0\right)  $.
Since the control $u$ takes its value in $\left[  0,1\right]  $ which also
contains $\theta_{1}$, $\frac{d\theta}{dt}$ is negative and the result is
obtained. For all points on $F_{2}$, $n$ can be choosen as $\left(
0,1,0\right)  $. Thanks to definition of $\theta_{2}$ and $\eta$, $\frac
{dv}{dt}$ is negative and the result is obtained. For all points on $F_{3}$,
$n$ can be choosen as $\left(  0,-1,1\right)  $. $\frac{dv_{r}}{dt}$ is zero,
and consequently $F_{3}$ is positively invariant.
\end{proof}

The invariance of the set $F_{3}$ is biologically plausible, since once the
fruit is totally rotten it remain definitely in that state, the fruit is lost.
The set $BS$ is also reasonable for biological reasons: the inhibition rate is
bounded, the rotten volume is no larger than the total volume, and fungus
attack reduces the size of a mature fruit.

\subsection{Optimal control of the within-host model\label{WithinHostModel3}}

\qquad In this subsection we apply control to model (\ref{ModelIntraHote}),
which we repeat here for convenience:
\begin{equation}
\left\{
\begin{array}
[c]{l}%
d\theta/dt=\alpha\left(  t,\theta\right)  \left(  1-\theta/\left(
1-\theta_{_{1}}u\left(  t\right)  \right)  \right) \\
dv/dt=\beta\left(  t,\theta\right)  \left(  1-v\theta_{_{2}}/\left(  \left(
1-\theta\right)  \eta\left(  t\right)  v_{\max}\right)  \right) \\
dv_{r}/dt=\gamma\left(  t,\theta\right)  \left(  1-v_{r}/v\right)
\end{array}
\right.
\end{equation}
For the control problem we focus on the first equation. This equation is
controllable in $\left]  0,1\right[  $ since $\theta$ is continuous and
$1-\theta_{_{1}}u\left(  t\right)  $ is an asymptotic threshold which can be
set easily. Giving a time $T>0$ (for example the annual production duration)
we search for $u$ in $L_{loc}^{2}\left(
%TCIMACRO{\U{211d} }%
%BeginExpansion
\mathbb{R}
%EndExpansion
_{+},\left[  0,1\right]  \right)  $ such that the following functional
(previously used in \cite{aldila,emvudu}) is minimized:
\[
J_{T}\left(  u\right)  =\int\nolimits_{0}^{T}\left(  ku^{2}\left(  t\right)
+\theta^{2}\left(  t\right)  \right)  dt+f\left(  \theta\left(  T\right)
\right)  ,
\]
where $k>0$ can be interpreted as the cost ratio related to the use of control
effort $u$. This functional reflects the fact that reducing inhibition rate
$\theta$ will lead to increased fruit production (larger volumes with a
relatively lower level of infection), while minimizing $u$ will reduce
financial and environmental costs. We use the squares of $u$ and $\theta$ in
the integrand because this choice facilitates the technical calculations
required for minimization.

We note in passing that we could had tried to minimize the more practically
relevant expression :
\begin{align*}
& \int\nolimits_{0}^{T}\left(  ku\left(  t\right)  +\theta\left(  t\right)
+\left(  v_{\max}-v\left(  t\right)  \right)  +v_{r}\left(  t\right)  \right)
dt\\
& \quad+\theta\left(  T\right)  +\left(  v_{\max}-v\left(  T\right)  \right)
+v_{r}\left(  T\right)
\end{align*}
However, an exact computation of this functional would require precise
expressions for $\alpha,\beta,\gamma,\eta$ in the system (\ref{ModelIntraHote}%
). As far as the authors know, there is no previous study which gives those
parameters. It seemed more advantageous to us to limit the random choice of
parameters, so that we could perform representative simulations.

We define the set
\[
U^{K}\equiv\left\{
\begin{array}
[c]{c}%
u\in C\left(  \left[  0,T\right]  ;\left[  0,1\right]  \right)  ;\forall
t,s\in\left[  0,T\right]  ,\text{ }\\
\left\vert u\left(  t\right)  -u\left(  s\right)  \right\vert \leq K\left\vert
t-s\right\vert
\end{array}
\right\}  ,
\]
which is nonempty for every $K\geq0$.

\begin{theorem}
Let $K\geq0$. There is a control $u^{\ast}\in U^{K}$ which minimizes the cost
$J_{T}$\textbf{.}
\end{theorem}

\begin{proof}
Since $J_{T}\geq0$ it is bounded below. Let the infinimum be $J^{\ast}$, and
let $\left(  u_{n}\right)  _{n\in%
%TCIMACRO{\U{2115} }%
%BeginExpansion
\mathbb{N}
%EndExpansion
}$ be a sequence in $U^{K}$ such that $\left(  J_{T}\left(  u_{n}\right)
\right)  _{n\in%
%TCIMACRO{\U{2115} }%
%BeginExpansion
\mathbb{N}
%EndExpansion
}$ converges to $J^{\ast}$. The definition of $U^{K}$ implies that $\left(
u_{n}\right)  _{n\in%
%TCIMACRO{\U{2115} }%
%BeginExpansion
\mathbb{N}
%EndExpansion
}$ is bounded and uniformly equicontinuous on $\left[  0,T\right]  $. By the
Ascoli theorem, there is a subsequence $\left(  u_{n_{k}}\right)  $ which
converges to a control $u^{\ast}$. Since the cost function is continuous with
respect to $u$ it follows that $J_{T}\left(  u^{\ast}\right)  =J^{\ast}. $
\end{proof}

\begin{theorem}
Suppose that $\alpha$ depends only on time. If there is an optimal control
strategy $u$ which minimizes $J_{T}$, then $u$ is unique and satisfies
\begin{equation}
u\left(  t\right)  =\left\{
\begin{array}
[c]{l}%
1\text{ when }27\alpha\theta_{_{1}}^{2}\theta p\geq8k\\
\frac{w_{3}\left(  t\right)  -1}{\theta_{_{1}}w_{3}\left(  t\right)  }\text{
when }27\alpha\theta_{_{1}}^{2}\theta p<8k
\end{array}
\right. \label{Optimalu}%
\end{equation}
where $w_{3}\left(  t\right)  $ is the element of $\left[  1,\min\left\{
3/2,1/\left(  1-\theta_{_{1}}\right)  \right\}  \right]  $ which is the
nearest to the smallest nonnegative solution of the equation $\alpha
\theta_{_{1}}^{2}\theta pw^{3}-2kw+2k=0$ and
\begin{equation}
\left\{
\begin{array}
[c]{l}%
d\theta/dt=\alpha\left(  t,\theta\right)  \left(  1-\theta/\left(
1-\theta_{_{1}}u\left(  t\right)  \right)  \right) \\
dp/dt=\alpha\left(  t\right)  p\left(  t\right)  /\left(  1-\theta_{_{1}%
}u\left(  t\right)  \right)  -2\theta\\
\theta\left(  0\right)  =\theta_{0}\text{, }\theta\left(  T\right)
=\theta_{T}\text{, }p(T)=f^{^{\prime}}\left(  \theta_{T}\right)
\end{array}
\right. \label{OptimumSystemOde1}%
\end{equation}

\end{theorem}

\begin{proof}
According the maximum principle, minimizing $J_{T}$ is equivalent to
minimizing the Hamiltonian functional
\begin{align*}
H\left(  t,\theta,u\right)   &  =ku^{2}\left(  t\right)  +\theta^{2}\left(
t\right)  +f\left(  \theta\left(  T\right)  \right) \\
&  +\alpha\left(  t\right)  p\left(  t\right)  \left(  1-\theta_{_{1}}u\left(
t\right)  -\theta\right)  /\left(  1-\theta_{_{1}}u\left(  t\right)  \right)
\end{align*}
where the adjoint state $p$ is the solution to the following problem
\begin{equation}
\left\{
\begin{array}
[c]{l}%
dp/dt=\alpha\left(  t\right)  p/\left(  1-\theta_{_{1}}u\left(  t\right)
\right)  -2\theta\\
p(T)=f^{^{\prime}}\left(  \theta\left(  T\right)  \right)
\end{array}
\right. \label{PBAdjoint}%
\end{equation}
To simplify the expression, let
\begin{equation}
\label{DefOfw}w\equiv1/\left(  1-\theta_{_{1}}u\right)  \in\left[  1,1/\left(
1-\theta_{_{1}}u\right)  \right] .
\end{equation}
Then the new equivalent functional to minimize is
\[
J_{T}^{1}\left(  w\right)  =\int\nolimits_{0}^{T}\left(  k\left(
\frac{w\left(  t\right)  -1}{\theta_{_{1}}w\left(  t\right)  }\right)
^{2}+\theta^{2}\left(  t\right)  \right)  dt+f\left(  \theta\left(  T\right)
\right)
\]
$\partial_{w}H=0$ if and only if
\begin{equation}
\label{partialH}\alpha\theta_{_{1}}^{2}\theta pw^{3}-2kw+2k=0.
\end{equation}
This equation has a unique nonpositive solution when $27\alpha\theta_{_{1}%
}^{2}\theta p\geq8k$. It has at least one nonnegative solution when
$27\alpha\theta_{_{1}}^{2}\theta p<8k$. We can choose $w$ in the following
way:
\[
w\left(  t\right)  =\left\{
\begin{array}
[c]{l}%
\frac{1}{1-\theta_{_{1}}}\text{ when }27\alpha\theta_{_{1}}^{2}\theta
p\geq8k\\
w_{3}\left(  t\right)  \text{ when }27\alpha\theta_{_{1}}^{2}\theta p<8k
\end{array}
\right.
\]
where $w_{3}\left(  t\right)  $ is the element of $\left[  1,\min\left\{
3/2,1/\left(  1-\theta_{_{1}}\right)  \right\}  \right]  $ that is the nearest
to the smallest nonnegative solution of (\ref{partialH}). It follows from the
definition of $w$ in (\ref{DefOfw}) and algebraic rearrangement that the
optimal control $u$ is given by (\ref{Optimalu}),
%Therefore $u$ is
%given such as in the theorem.
%\begin{equation*}
%u\left( t\right) =\left\{
%\begin{array}{l}
%1\text{ when }27\alpha \theta _{_{1}}^{2}\theta p\leq 8k \\
%\frac{w_{3}\left( t\right) -1}{\theta _{_{1}}w_{3}\left( t\right) }\text{
%when }27\alpha \theta _{_{1}}^{2}\theta p>8k%
%\end{array}%
%\right.
%\end{equation*}%
where $(p,\theta)$ is a solution to the system $\left(
\ref{OptimumSystemOde1}\right)  $. The uniqueness of $u$ follows from the
uniqueness of the solution of the system $\left(  \ref{OptimumSystemOde1}%
\right)  $.
\end{proof}

\subsection{Computer simulations of the controlled within-host
model\label{WithinHostModel4}}

We performed simulations in order to demonstrate the practical controllability
of the system, For these simulations we used an inhibition pressure of the
form
\[
\alpha\left(  t\right)  =a\left(  t-b\right)  ^{2}\left(  1-\cos\left(  2\pi
t/c\right)  \right)  ,
\]
with $b$ and $c$ in $\left[  0,1\right]  $. This function reflects the
seasonality of empirically-based severity index models found in the literature
\cite{danneberger,dodd,duthie}. The particular values used in the simulation
were $a=4,b=0.75$, $c=0.2$ and $k=1$. We also took $f\left(  \theta\left(
T\right)  \right)  =\theta\left(  T\right)  $, so that $p\left(  T\right)
=f^{\prime}\left(  \theta\left(  T\right)  \right)  =1$. In this case and the
shooting method can be used to numerically estimate the value of $p_{0}$ which
produces a solution to $\left(  \ref{PBAdjoint}\right)  $.

\begin{center}
\bigskip%

%TCIMACRO{\FRAME{itbpFU}{3.2076in}{2.4284in}{0in}{\Qcb{$\QTR{bf}{Fig1}$:
%Optimum control effort over a one-year period $\theta_{0}=0.2,\theta
%_{1}=1-0.4$.}}{\Qlb{Control1}}{Figure}{\special{ language "Scientific Word";
%type "GRAPHIC";  maintain-aspect-ratio TRUE;  display "USEDEF";
%valid_file "T";  width 3.2076in;  height 2.4284in;  depth 0in;
%original-width 4.6752in;  original-height 3.5336in;  cropleft "0";
%croptop "1";  cropright "1";  cropbottom "0";
%tempfilename 'N2OTHI00.wmf';tempfile-properties "XPR";}} }%
%BeginExpansion
{\parbox[b]{3.2076in}{\begin{center}
\includegraphics[
natheight=3.533600in,
natwidth=4.675200in,
height=2.4284in,
width=3.2076in
]%
{N2OTHI00.jpg}%
\\
$\mathbf{Fig1}$: Optimum control effort over a one-year period $\theta
_{0}=0.2,\theta_{1}=1-0.4$.
\end{center}}}
%EndExpansion
%

%TCIMACRO{\FRAME{itbpFU}{3.1756in}{2.3999in}{0in}{\Qcb{$\QTR{bf}{Fig2}$:
%Evolution of inhibition rate over a one-year period with $\theta
%_{0}=0.2,\theta_{1}=1-0.4$.}}{\Qlb{InhibitionRate1}}{Figure}%
%{\special{ language "Scientific Word";  type "GRAPHIC";
%maintain-aspect-ratio TRUE;  display "USEDEF";  valid_file "T";
%width 3.1756in;  height 2.3999in;  depth 0in;  original-width 4.6847in;
%original-height 3.5336in;  cropleft "0";  croptop "1";  cropright "1";
%cropbottom "0";  tempfilename 'N2OTMV01.wmf';tempfile-properties "XPR";}} }%
%BeginExpansion
{\parbox[b]{3.1756in}{\begin{center}
\includegraphics[
natheight=3.533600in,
natwidth=4.684700in,
height=2.3999in,
width=3.1756in
]%
{N2OTMV01.jpg}%
\\
$\mathbf{Fig2}$: Evolution of inhibition rate over a one-year period with
$\theta_{0}=0.2,\theta_{1}=1-0.4$.
\end{center}}}
%EndExpansion
%

%TCIMACRO{\FRAME{itbpFU}{3.2906in}{2.4924in}{0in}{\Qcb{$\QTR{bf}{Fig3}$:
%Optimum control effort over a one-year period $\theta_{0}=0.5,\theta
%_{1}=1-0.4$.}}{\Qlb{Control2}}{Figure}{\special{ language "Scientific Word";
%type "GRAPHIC";  maintain-aspect-ratio TRUE;  display "USEDEF";
%valid_file "T";  width 3.2906in;  height 2.4924in;  depth 0in;
%original-width 4.6752in;  original-height 3.5336in;  cropleft "0";
%croptop "1";  cropright "1";  cropbottom "0";
%tempfilename 'N2OTP202.wmf';tempfile-properties "XPR";}} }%
%BeginExpansion
{\parbox[b]{3.2906in}{\begin{center}
\includegraphics[
natheight=3.533600in,
natwidth=4.675200in,
height=2.4924in,
width=3.2906in
]%
{N2OTP202.jpg}%
\\
$\mathbf{Fig3}$: Optimum control effort over a one-year period $\theta
_{0}=0.5,\theta_{1}=1-0.4$.
\end{center}}}
%EndExpansion
%

%TCIMACRO{\FRAME{itbpFU}{3.3641in}{2.5486in}{0in}{\Qcb{$\QTR{bf}{Fig4}$:
%Evolution of inhibition rate over a one-year period $\theta_{0}=0.5,\theta
%_{1}=1-0.4$.}}{\Qlb{InhibitionRate2}}{Figure}%
%{\special{ language "Scientific Word";  type "GRAPHIC";
%maintain-aspect-ratio TRUE;  display "USEDEF";  valid_file "T";
%width 3.3641in;  height 2.5486in;  depth 0in;  original-width 4.6847in;
%original-height 3.5431in;  cropleft "0";  croptop "1";  cropright "1";
%cropbottom "0";  tempfilename 'N2OTPQ03.wmf';tempfile-properties "XPR";}} }%
%BeginExpansion
{\parbox[b]{3.3641in}{\begin{center}
\includegraphics[
natheight=3.543100in,
natwidth=4.684700in,
height=2.5486in,
width=3.3641in
]%
{N2OTPQ03.jpg}%
\\
$\mathbf{Fig4}$: Evolution of inhibition rate over a one-year period
$\theta_{0}=0.5,\theta_{1}=1-0.4$.
\end{center}}}
%EndExpansion

\end{center}

The above figures show how the control strategy adapts itself in response to
inhibition pressure represented by $\alpha$. Figures 1-2 correspond to the
case of low initial effective inhibition rate ($\theta_{0}=0.2$, corresponding
to a low initial level of infection), while Figures 3-4 correspond to the case
of high initial effective inhibition rate ($\theta_{0}=0.5$). The simulations
also show in different cases the effectiveness of the optimal strategy as
compared to taking no control action. Regardless of whether the initial
inhibition rate is above or below the threshold $1-\theta_{1}$, the dynamics
are sensitive to the control effort.

\section{A diffusion model\label{sec:diffusion}}

In this section, we will model the geographical spread of the disease via
diffusive factors such as the movement of inoculum through ground water and wind.

\subsection{Specification of the diffusion model\label{DiffusionModel1}}

We include the effect of diffusive factors on the spread of infection by
adding a diffusion term to the within-host equation for $\theta$ from system
(\ref{ModelIntraHote}). Together with boundary conditions, the model is
\begin{align}
&  \partial\theta/\partial t=\alpha\left(  t,x,\theta\right)  \left(
1-\theta/\left(  1-\theta_{_{1}}u\left(  t,x\right)  \right)  \right)
\nonumber\\
&  \qquad\quad~~+\operatorname{div}\left(  A\left(  t,x,\theta\right)
\nabla\theta\right)  ~~\text{ on }%
%TCIMACRO{\U{211d} }%
%BeginExpansion
\mathbb{R}
%EndExpansion
_{+}^{\ast}\times\Omega,\label{ModelIntraHoteDiff1}\\
&  \left\langle A\left(  t,x,\theta\right)  \nabla\theta,n\right\rangle
=0\qquad\quad\quad\text{ on }%
%TCIMACRO{\U{211d} }%
%BeginExpansion
\mathbb{R}
%EndExpansion
_{+}^{\ast}\times\partial\Omega,\label{ModelIntraHoteDiff2}\\
&  \theta\left(  0,x\right)  =\theta_{0}\left(  x\right)  \geq0\qquad
\qquad\quad x\in\overline{\Omega},\label{ModelIntraHoteDiff3}%
\end{align}
where $\Omega$ is an open bounded subset of $%
%TCIMACRO{\U{211d} }%
%BeginExpansion
\mathbb{R}
%EndExpansion
^{3}$ with a continuously differentiable boundary $\partial\Omega$, and
$\theta_{_{1}}\in\left[  0,1\right[  $. For a given element $\left(
t,x,\theta\right)  $, $A$ is a $3\times3$-matrix $\left(  a_{ij}\left(
t,x,\theta\right)  \right)  $. The functions $\alpha$ and $a_{ij}$ are assumed
to be nonnegative.\textbf{\ }Since $\theta$ depends on $\left(  t,x\right)  $,
the functions $\alpha$ and $a_{ij}$ can be identified with elements of the set
$C\left(  \left[  0,T\right]  ;H^{1}\left(  \Omega\right)  \right)  $. The
function $u\in C\left(  \left[  0,T\right]  ;H^{1}\left(  \Omega\right)
\right)  $ designates the control, which takes its values in the set $\left[
0,1\right]  $. $\left(  \ref{ModelIntraHoteDiff2}\right)  $ may be interpreted
as a dependence of the flux of inoculum with respect to diffusion factors. In
particular, when $A$ is the identity matrix $\left(
\ref{ModelIntraHoteDiff12}\right)  $ could be interpreted as there is no flux
between exterior and interior of the domain $\Omega$.

Practically, time can be subdivided into intervals on which parameters are
approximately constant. We may thus study the system on each interval
separately, and assume that all parameters are constant. We also assume that
functions $\alpha$ and $A$ do not depend on $\theta$. This leads to the
following simplified model,%

\begin{align}
& \partial\theta/\partial t =\alpha\left(  x\right)  \left(  1-\theta/\left(
1-\theta_{_{1}}u\left(  x\right)  \right)  \right) \nonumber\\
& \qquad\quad~~+\operatorname{div} \left(  A\left(  x\right)  \nabla
\theta\right)  ~~ \text{ on }\left]  0,T\right[  \times\Omega
,\label{ModelIntraHoteDiff12}\\
& \left\langle A\left(  x\right)  \nabla\theta,n\right\rangle =0\qquad\qquad~
\text{ on } \left]  0,T\right[  \times\partial\Omega
,\label{ModelIntraHoteDiff22}\\
& \theta\left(  0,x\right)  =\theta_{0}\left(  x\right)  \geq0\qquad\qquad
x\in\overline{\Omega},\label{ModelIntraHoteDiff32}%
\end{align}

In order to formalize the model, we define the Hilbert space
\[
E=\left\{  \theta\in H^{2}\left(  \Omega\right)  ;\text{ }\theta\text{
satisfies }\left(  \ref{ModelIntraHoteDiff22}\right)  \right\}
\]
provided with the inner product
\[
\left\langle f,g\right\rangle _{E}=\int\nolimits_{U}\left(  fg+\left\langle
\nabla f,\nabla g\right\rangle +\Delta f.\Delta g\right)  dx
\]
Define also the linear unbounded operator $\pounds :D\left(  \pounds \right)
=E\subset L^{2}\left(  \Omega\right)  \rightarrow L^{2}\left(  \Omega\right)
$ as
\[
\pounds \theta=\frac{\alpha\theta}{1-\theta_{_{1}}u}-\operatorname{div}\left(
A\left(  x\right)  \nabla\theta\right)
\]
Then equation $\left(  \ref{ModelIntraHoteDiff12}\right)  $ takes the
following form
\begin{equation}
\partial\theta/\partial t+\pounds \theta=\alpha.\label{ModelIntraHoteDiff12A}%
\end{equation}
We also introduce the following condition, which we will use to ensure that
the system has realistic solutions: \medskip

$\left(  H4\right)  $ There exists a constant $C>0$ such that for almost every
$x\in\Omega$, $A\left(  x\right)  $ is symmetric, positive definite and
\[
\left\langle v,A\left(  x\right)  v\right\rangle \geq C\left\langle
v,v\right\rangle ,\text{ }\forall v.
\]

\subsection{Well-posedness of the diffusion model\label{DiffusionModel2}}

We are now ready to prove that our model has been a mathematically and
epidemiologically well-posed. In other words, we show that exists a unique
solution $0 \le\theta(t,x) \le$ to the system $\left(
\ref{ModelIntraHoteDiff12}\right)  -\left(  \ref{ModelIntraHoteDiff32}\right)
$. This shall follow from the Hille-Yosida theorem\footnote{See \cite{brezis}
p 185.}: but first we need the following proposition.

\begin{proposition}
\label{PropMaxMon}If $A\left(  x\right)  $ is a positive semidefinite bilinear
form for almost every $x\in\Omega$, then the linear operator $\pounds $ is
monotone on $E$. Moreover, if $A\left(  x\right)  $ satisfies ($H4$), then
$\pounds $ is maximal.
\end{proposition}

\begin{proof}
\begin{enumerate}
\item[$\left(  ii\right)  $] To show $\pounds $ is monotone, we let $\theta\in
E$ and compute:
\begin{align*}
&  \int\nolimits_{\Omega}\pounds \theta\times\theta\,dx\\
&  =\int\nolimits_{\Omega}\left(  \alpha\theta^{2}/\left(  1-\theta_{_{1}%
}u\right)  \right)  dx-\int\nolimits_{\Omega}\operatorname{div}\left(
A\left(  x\right)  \nabla\theta\right)  \theta dx\\
&  =\left(  1/\left(  1-\theta_{_{1}}u\right)  \right)  \left\Vert \theta
\sqrt{\alpha}\right\Vert _{L^{2}\left(  \Omega\right)  }^{2}-\int
\nolimits_{\Omega}\operatorname{div}\left(  A\left(  x\right)  \nabla
\theta\right)  \theta dx\\
&  \geq-\int\nolimits_{\partial\Omega}\left\langle A\left(  x\right)
\nabla\theta,n\left(  x\right)  \right\rangle \theta dx+\int\nolimits_{\Omega
}\left\langle A\left(  x\right)  \nabla\theta,\nabla\theta\right\rangle dx\\
&  \geq0.
\end{align*}

\item[$\left(  ii\right)  $] To show $\pounds $ is maximal , we let $f\in
L^{2}\left(  \Omega\right)  $ and seek $\theta\in E$ such that $\theta
+\pounds \theta=f$. Given $\varphi\in E$, we have
\begin{align*}
&  \int\nolimits_{\Omega}\left(  \theta+\pounds \theta\right)  \times\varphi
dx\\
&  =\int\nolimits_{\Omega}\varphi\theta\left(  1+\alpha-\theta_{_{1}}u\right)
/\left(  1-\theta_{_{1}}u\right)  dx\\
&  -\int\nolimits_{\Omega}\operatorname{div}\left(  A\left(  x\right)
\nabla\theta\right)  \varphi dx\\
&  =\int\nolimits_{\Omega}\varphi\theta\left(  1+\alpha-\theta_{_{1}}u\right)
/\left(  1-\theta_{_{1}}u\right) \\
&  +\int\nolimits_{\Omega}\left\langle A\left(  x\right)  \nabla\theta
,\nabla\varphi\right\rangle dx\\
&  -\int\nolimits_{\partial\Omega}\left\langle A\left(  x\right)  \nabla
\theta,n\left(  x\right)  \right\rangle \varphi dx\\
&  =\int\nolimits_{\Omega}\varphi\theta\left(  1+\alpha-\theta_{_{1}}u\right)
/\left(  1-\theta_{_{1}}u\right) \\
&  +\int\nolimits_{\Omega}\left\langle A\left(  x\right)  \nabla\theta
,\nabla\varphi\right\rangle dx\\
&  \equiv p\left(  \theta,\varphi\right) ,
\end{align*}
where $p$ is a symmetric continuous and coercive bilinear form on
$H^{1}\left(  \Omega\right)  $. The Lax-Milgram theorem\footnote{See
\cite{brezis}} implies that there is a unique $\theta\in H^{1}\left(
\Omega\right)  $ such that $\theta+\pounds \theta=f$. Using regularization
methods similar those used in Theorem 9.26 of \cite{brezis}, it follows that
$\theta\in E$.
\end{enumerate}
\end{proof}

Given that the linear operator $\pounds $ is maximal monotone and $\theta_{0}$
is in $E$, then by the Hille-Yosida theorem there is a unique function
$\theta\in C^{1}\left(  \left[  0,T\right]  ;L^{2}\left(  \Omega\right)
\right)  \bigcap C\left(  \left[  0,T\right]  ;E\right)  $ which satisfies
$\left(  \ref{ModelIntraHoteDiff12}\right)  -\left(
\ref{ModelIntraHoteDiff32}\right)  $, and $\forall\left(  t,x\right)
\in\left[  0,T\right]  \times\Omega$ we have
\[
\theta\left(  t,x\right)  =\left(  S_{\pounds }\left(  t\right)  \theta
_{0}\right)  \left(  x\right)  +\int\nolimits_{0}^{t}\left(  S_{\pounds }%
\left(  t-s\right)  \alpha\right)  \left(  x\right)  ds,\text{ }%
\]
where $S_{\pounds }\left(  t\right)  $ is the contraction semigroup generated
by $-\pounds $.

Now that we have established existence and uniqueness of the solution $\theta
$, we now prove that $0 \le\theta(t,x) \le1$ for all $(t,x)$ in the domain.
assuming that $A(x)$ satisfies the condition ($H4$). We define
\begin{align*}
m &  \equiv\underset{\partial\Omega}{\inf}\,\theta_{0},\\
M &  \equiv\max\left\{  \underset{\partial\Omega}{\sup}\,\theta_{0}%
,~\underset{\Omega}{\sup}\left(  1-\theta_{_{1}}u\right)  \right\}  ,\\
v &  \equiv1/\left(  1-\theta_{_{1}}u\right)  .
\end{align*}
Note that $M \le1$ as long as $\theta_{0} \le1$ and $0 \le\theta_{1} u \le1$.

Let $E_{+}$ designate the set of elements in $E$ which are nonnegative almost
everywhere on $\overline{\Omega}$. The following theorem gives sufficient
conditions under which the solution $\theta$ of $\left(
\ref{ModelIntraHoteDiff12}\right)  -\left(  \ref{ModelIntraHoteDiff32}\right)
$\ is bounded by $M \le1$ and the positive cone $E_{+}$ of $E$ is positively invariant.

\begin{theorem}
\label{TheoMinMax}If $A\left(  x\right)  \nabla e^{tv\alpha}=0$ for every
$\left(  t,x\right)  \in$\ $\left[  0,T\right]  \times\Omega$ then for almost
every $x$ in $\Omega$,
\begin{equation}
m\leq e^{tv\alpha}\theta\left(  t,x\right)  \text{, }t\in\left[  0,T\right]
.\label{Min}%
\end{equation}
Moreover if $A\left(  x\right)  \nabla v=0$ for every $\left(  t,x\right)
\in$\ $\left[  0,T\right]  \times\Omega$, then
\begin{equation}
\theta\left(  t,x\right)  \leq M\label{Max}%
\end{equation}
In particular, $\left(  \ref{Min}\right)  $ and $\left(  \ref{Max}\right)  $
hold when $\alpha$ and $u$ do not depend on the space variable $x$.
\end{theorem}

\begin{proof}
Let $G\in C^{1}\left(
%TCIMACRO{\U{211d} }%
%BeginExpansion
\mathbb{R}
%EndExpansion
\right)  $ such that

\begin{enumerate}
\item[$\left(  i\right)  $] $G\left(  s\right)  =0,$ $\forall s\leq0$, and

\item[$\left(  ii\right)  $] $0<G^{\prime}\left(  s\right)  \leq C,$ $\forall
s>0 $.
\end{enumerate}

Define
\begin{align*}
H\left(  s\right)   & \equiv\int\nolimits_{0}^{s}G\left(  \sigma\right)
d\sigma,\forall s\in\mathbb{R},\\
\varphi_{1}\left(  t\right)   & \equiv\int\nolimits_{\Omega}H\left(
m-e^{tv\alpha}\theta\left(  t,x\right)  \right)  dx,\\
\varphi_{2}\left(  t\right)   & \equiv\int\nolimits_{\Omega}H\left(
e^{tv\alpha}\left(  \theta\left(  t,x\right)  -1/v\right)  \right)  dx.
\end{align*}
We observe that
\begin{align*}
& \varphi_{1},\varphi_{2}\in C\left(  \left[  0,T\right]  ; \mathbb{R}
\right)  \bigcap C^{1}\left(  \left]  0,T\right]  ; \mathbb{R}\right)  ,\\
& \varphi_{1},\varphi_{2}\geq0 \text{ on } \left[  0,T\right] ,\\
& \varphi_{1}\left(  0\right)  =\varphi_{2}\left(  0\right)  =0.
\end{align*}
We may also compute
\begin{align*}
&  \varphi_{1}^{\prime}\left(  t\right) \\
& ~ =-\int\nolimits_{\Omega}e^{tv\alpha}G\left(  m-e^{tv\alpha}\theta\right)
\left(  v\alpha\theta+\partial\theta/\partial t\right)  dx\\
& ~ =-\int\nolimits_{\Omega}e^{tv\alpha}G\left(  m-e^{tv\alpha}\theta\right)
\left(  \alpha-\pounds \theta+v\alpha\theta\right)  dx\\
& ~ =-\int\nolimits_{\Omega}\alpha G\left(  m-e^{tv\alpha}\theta\right)  dx\\
& \qquad+\int\nolimits_{\Omega}\left\langle A\left(  x\right)  \nabla
\theta,\nabla e^{tv\alpha}G\left(  m-e^{tv\alpha}\theta\right)  \right\rangle
dx\\
& ~ =-\int\nolimits_{\Omega}\alpha G\left(  m-e^{tv\alpha}\theta\right)  dx\\
& \qquad-\int\nolimits_{\Omega}e^{2tv\alpha}G^{\prime}\left(  m-e^{tv\alpha
}\theta\right)  \left\langle A\left(  x\right)  \nabla\theta,\nabla
\theta\right\rangle dx\\
& \qquad+\int\nolimits_{\Omega}\left(  G\left(  m-e^{tv\alpha}\theta\right)
-e^{tv\alpha}\theta G^{\prime}\left(  m-e^{tv\alpha}\theta\right)  \right) \\
& \qquad\qquad~~\times\left\langle A\left(  x\right)  \nabla\theta,\nabla
e^{tv\alpha}\right\rangle dx\\
& ~ \leq0.
\end{align*}
Since $\varphi_{1}^{\prime}\leq0$ on $%
%TCIMACRO{\U{211d} }%
%BeginExpansion
\mathbb{R}
%EndExpansion
_{+}^{\ast}$, $\varphi_{1}$ is identically zero on $%
%TCIMACRO{\U{211d} }%
%BeginExpansion
\mathbb{R}
%EndExpansion
_{+}$ and therefore almost everywhere in $\Omega$.
\[
m\leq e^{tv\alpha}\theta\left(  t,x\right)
\]
If $A\left(  x\right)  \nabla v=0$ for every $\left(  t,x\right)  \in
$\ $\left[  0,T\right]  \times\Omega$, then
\begin{align*}
&  \varphi_{2}^{\prime}\left(  t\right) \\
&  =\int\nolimits_{\Omega}G\left(  e^{tv\alpha}\left(  \theta-1/v\right)
\right)  \left(  -\alpha+v\alpha\theta+\partial\theta/\partial t\right)  dx\\
&  =\int\nolimits_{\Omega}e^{tv\alpha}G\left(  e^{tv\alpha}\left(
\theta-1/v\right)  \right)  \left(  -\pounds \theta+v\alpha\theta\right)  dx\\
&  =-\int\nolimits_{\Omega}\left\langle A\left(  x\right)  \nabla\theta,\nabla
e^{tv\alpha}G\left(  e^{tv\alpha}\left(  \theta-1/v\right)  \right)
\right\rangle dx\\
&  =-\int\nolimits_{\Omega}e^{2tv\alpha}G^{\prime}\left(  e^{tv\alpha}\left(
\theta-1/v\right)  \right)  \left\langle A\left(  x\right)  \nabla
\theta,\nabla\theta\right\rangle dx\\
&  \quad~ -\int\nolimits_{\Omega}\left(  e^{2tv\alpha}/v^{2}\right)
G^{\prime}\left(  e^{tv\alpha}\left(  \theta-1/v\right)  \right)  \left\langle
A\left(  x\right)  \nabla\theta,\nabla v\right\rangle dx\\
&  \quad~ -\int\nolimits_{\Omega}G\left(  e^{tv\alpha}\left(  \theta
-1/v\right)  \right)  \left\langle A\left(  x\right)  \nabla\theta,\nabla
e^{tv\alpha}\right\rangle dx\\
&  \quad~ -\int\nolimits_{\Omega}e^{tv\alpha}\left(  \theta-1/v\right)
G^{\prime}\left(  e^{tv\alpha}\left(  \theta-1/v\right)  \right) \\
&  \qquad\qquad~~ \times\left\langle A\left(  x\right)  \nabla\theta,\nabla
e^{tv\alpha}\right\rangle dx\\
&  \leq0.
\end{align*}
Since $\varphi_{2}^{\prime}\leq0$ on $%
%TCIMACRO{\U{211d} }%
%BeginExpansion
\mathbb{R}
%EndExpansion
_{+}^{\ast}$, $\varphi_{2}$ is identically zero on $%
%TCIMACRO{\U{211d} }%
%BeginExpansion
\mathbb{R}
%EndExpansion
_{+}$ and therefore almost everywhere in $\Omega$
\[
\theta\left(  t,x\right)  \leq M.
\]

\end{proof}

The following theorem proves boundedness of $\theta$ under more general conditions.

\begin{theorem}
\label{TheoMinMax2}Suppose that $v$ and $\alpha v$ are increasing functions $h
$ and $g$ of the state $\theta$, and there is a constant $K>0$ such that
\begin{equation}
ag^{\prime}\left(  \theta\right)  \leq K\left(  1+ag^{\prime}\left(
\theta\right)  \right)  \exp\left(  ag\left(  \theta\right)  \right)  ,\text{
}\forall a\geq0.\label{CondMinMax}%
\end{equation}
Then for every time $t\in\left[  0,T\right]  $ and almost every $x$ in
$\Omega$,
\begin{equation}
m\leq e^{tv\alpha}\theta\left(  t,x\right) \label{Min2}%
\end{equation}
and
\begin{equation}
\theta\left(  t,x\right)  \leq M.\label{Max2}%
\end{equation}

\end{theorem}

\begin{proof}
Let $G\in C^{1}\left(
%TCIMACRO{\U{211d} }%
%BeginExpansion
\mathbb{R}
%EndExpansion
\right)  $ such that

\begin{enumerate}
\item[$\left(  i\right)  $] $G\left(  s\right)  =0,$ $\forall s\leq0$, and

\item[$\left(  ii\right)  $] $KG\left(  s\right)  \leq G^{\prime}\left(
s\right)  \leq C,$ $\forall s>0$.
\end{enumerate}

Using $\left(  \ref{CondMinMax}\right)  $ and the fact that operator $A$ is
monotone, we have
\begin{align*}
\left\langle A\left(  x\right)  \nabla\theta,\nabla v\right\rangle  &
=\left\langle A\left(  x\right)  \nabla\theta,\nabla h\left(  \theta\right)
\right\rangle \\
&  =h^{\prime}\left(  \theta\right)  \left\langle A\left(  x\right)
\nabla\theta,\nabla\theta\right\rangle \\
&  \geq0,
\end{align*}%
\begin{align*}
\left\langle A\left(  x\right)  \nabla\theta,\nabla e^{tw\alpha}\right\rangle
&  =\left\langle A\left(  x\right)  \nabla\theta,\nabla\exp\left(  tg\left(
\theta\right)  \right)  \right\rangle \\
&  =tg^{\prime}\left(  \theta\right)  \exp\left(  tg\left(  \theta\right)
\right)  \left\langle A\left(  x\right)  \nabla\theta,\nabla\theta
\right\rangle \\
&  \geq0,
\end{align*}
and
\begin{align*}
&  \left\langle A\left(  x\right)  \nabla\theta,\nabla e^{tv\alpha}G\left(
m-e^{tv\alpha}\theta\right)  \right\rangle \\
&  =\left\langle A\left(  x\right)  \nabla\theta,\nabla\exp\left(  tg\left(
\theta\right)  \right)  G\left(  m-\theta\exp\left(  tg\left(  \theta\right)
\right)  \right)  \right\rangle \\
&  =tg^{\prime}\left(  \theta\right)  \exp\left(  tg\left(  \theta\right)
\right)  G\left(  m-\theta\exp\left(  tg\left(  \theta\right)  \right)
\right)  \left\langle A\left(  x\right)  \nabla\theta,\nabla\theta
\right\rangle \\
&  -\exp\left(  2tg\left(  \theta\right)  \right)  \left(  1+tg^{\prime
}\left(  \theta\right)  \right)  G^{\prime}\left(  m-\theta\exp\left(
tg\left(  \theta\right)  \right)  \right)  \left\langle A\left(  x\right)
\nabla\theta,\nabla\theta\right\rangle \\
&  \leq\left(  tg^{\prime}\left(  \theta\right)  -K\left(  1+tg^{\prime
}\left(  \theta\right)  \right)  \exp\left(  tg\left(  \theta\right)  \right)
\right)  \exp\left(  tg\left(  \theta\right)  \right) \\
&  \times G\left(  m-\theta\exp\left(  tg\left(  \theta\right)  \right)
\right)  \left\langle A\left(  x\right)  \nabla\theta,\nabla\theta
\right\rangle \\
&  \leq0.
\end{align*}
Define
\begin{align*}
H\left(  s\right)   & \equiv\int\nolimits_{0}^{s}G\left(  \sigma\right)
d\sigma,\forall s\in\mathbb{R},\\
\varphi_{1}\left(  t\right)   & \equiv\int\nolimits_{\Omega}H\left(
m-e^{tv\alpha}\theta\left(  t,x\right)  \right)  dx,\\
\varphi_{2}\left(  t\right)   & \equiv\int\nolimits_{\Omega}H\left(
e^{tv\alpha}\left(  \theta\left(  t,x\right)  -1/w\right)  \right)  dx.
\end{align*}

Note that as in Theorem \ref{TheoMinMax} we have
\begin{align*}
& \varphi_{1},\varphi_{2}\in C\left(  \left[  0,T\right]  ; \mathbb{R}
\right)  \bigcap C^{1}\left(  \left]  0,T\right]  ; \mathbb{R}\right)  ,\\
& \varphi_{1},\varphi_{2}\geq0 \text{ on } \left[  0,T\right] ,\\
& \varphi_{1}\left(  0\right)  =\varphi_{2}\left(  0\right)  =0.
\end{align*}
As in Theorem \ref{TheoMinMax} we may compute
\begin{align*}
&  \varphi_{1}^{\prime}\left(  t\right) \\
&  =-\int\nolimits_{\Omega}e^{tv\alpha}G\left(  m-e^{tv\alpha}\theta\right)
\left(  w\alpha\theta+\partial\theta/\partial t\right)  dx\\
&  =-\int\nolimits_{\Omega}e^{tv\alpha}G\left(  m-e^{tv\alpha}\theta\right)
\left(  \alpha-\pounds \theta+v\alpha\theta\right)  dx\\
&  =-\int\nolimits_{\Omega}\alpha G\left(  m-e^{tv\alpha}\theta\right)  dx\\
&  \quad+\int\nolimits_{\Omega}\left\langle A\left(  x\right)  \nabla
\theta,\nabla e^{tv\alpha}G\left(  m-e^{tv\alpha}\theta\right)  \right\rangle
dx\\
&  \leq\int\nolimits_{\Omega}\left\langle A\left(  x\right)  \nabla
\theta,\nabla e^{tv\alpha}G\left(  m-e^{tv\alpha}\theta\right)  \right\rangle
dx\\
&  \leq0.
\end{align*}
Since $\varphi_{1}^{\prime}\leq0$ on $%
%TCIMACRO{\U{211d} }%
%BeginExpansion
\mathbb{R}
%EndExpansion
_{+}^{\ast}$, $\varphi_{1}$ is identically null on $\left[  0,T\right]  $ and
therefore almost everywhere in $\Omega$
\[
m\leq e^{tv\alpha}\theta\left(  t,x\right) .
\]
We also have
\begin{align*}
&  \varphi_{2}^{\prime}\left(  t\right) \\
&  ~~=\int\nolimits_{\Omega}G\left(  e^{tv\alpha}\left(  \theta-1/v\right)
\right)  \left(  -\alpha+v\alpha\theta+\partial\theta/\partial t\right)  dx\\
&  ~~=\int\nolimits_{\Omega}e^{tv\alpha}G\left(  e^{tv\alpha}\left(
\theta-1/v\right)  \right)  \left(  -\pounds \theta+v\alpha\theta\right)  dx\\
&  ~~=-\int\nolimits_{\Omega}\left\langle A\left(  x\right)  \nabla
\theta,\nabla e^{tv\alpha}G\left(  e^{tv\alpha}\left(  \theta-1/v\right)
\right)  \right\rangle dx\\
&  ~~=-\int\nolimits_{\Omega}e^{2tv\alpha}G^{\prime}\left(  e^{tv\alpha
}\left(  \theta-1/v\right)  \right)  \left\langle A\left(  x\right)
\nabla\theta,\nabla\theta\right\rangle dx\\
&  \qquad-\int\nolimits_{\Omega}\left(  e^{2tv\alpha}/v^{2}\right)  G^{\prime
}\left(  e^{tv\alpha}\left(  \theta-1/v\right)  \right)  \left\langle A\left(
x\right)  \nabla\theta,\nabla v\right\rangle dx\\
&  \qquad-\int\nolimits_{\Omega}G\left(  e^{tv\alpha}\left(  \theta
-1/v\right)  \right)  \left\langle A\left(  x\right)  \nabla\theta,\nabla
e^{tv\alpha}\right\rangle dx\\
&  \qquad-\int\nolimits_{\Omega}e^{tv\alpha}\left(  \theta-1/w\right)
G^{\prime}\left(  e^{tv\alpha}\left(  \theta-1/v\right)  \right) \\
&  \qquad\qquad\qquad\times\left\langle A\left(  x\right)  \nabla\theta,\nabla
e^{tv\alpha}\right\rangle dx\\
&  ~~\leq0.
\end{align*}
Since $\varphi_{2}^{\prime}\leq0$ on $\left[  0,T\right]  \setminus\left\{
0\right\}  $, $\varphi_{2}$ is identically null on $\left[  0,T\right]  $ and
therefore almost everywhere in $\Omega$
\[
\theta\left(  t,x\right)  \leq M.
\]

\end{proof}

Condition $\left(  \ref{CondMinMax}\right)  $ of Theorem \ref{TheoMinMax2} is
satisfied in particular when $g\geq0$. Using the same arguments as in the
proof of Proposition \ref{PropMaxMon}, there is a unique equilibrium
$\theta^{\ast}$, for the system $\left(  \ref{ModelIntraHoteDiff12}\right)
-\left(  \ref{ModelIntraHoteDiff32}\right)  $. $\theta^{\ast}$ is
asymptotically stable if and only if all the eigenvalues of the linear
operator $\pounds $ have nonnegative real parts. Stability of the equilibrium
$\theta^{\ast}$ has the advantage that the disease inhibition is maintained in
its neighborhood, which enables easier control strategies. In particular, the
norm of $\theta^{\ast}$ is a decreasing function of the control $u$.

\begin{proposition}
\label{PropositionSpectre}The real number $\lambda$ is not an eigenvalue of
$\pounds $ if at least one of the following conditions is satisfied:

\begin{enumerate}
\item[$\left(  i\right)  $] $\alpha\geq\lambda\left(  1-\theta_{_{1}}u\right)
$ almost everywhere in $\Omega$ and that inequality is strict on an
nonnegligible subset of $\Omega$.

\item[$\left(  ii\right)  $] There exists a real $k \ge0$ such that for every
$\theta\in E$
\[
\int\nolimits_{\Omega}\left\langle A\left(  x\right)  \nabla\theta
,\nabla\theta\right\rangle dx\geq k\left\Vert \theta\right\Vert _{H^{2}\left(
\Omega\right)  },
\]
and
\[
\left(  \alpha-\lambda\left(  1-\theta_{_{1}}u\right)  \right)  /\left(
1-\theta_{_{1}}u\right)  >-k
\]
almost everywhere in $\Omega$.
\end{enumerate}
\end{proposition}

\begin{proof}
Let $\theta,\varphi\in E$. Then we may compute%
\begin{align*}
&  \int\nolimits_{\Omega}\left(  \pounds \theta-\lambda\theta\right)
\times\varphi dx\\
&  ~~=\int\nolimits_{\Omega}\varphi\theta\left(  \alpha-\lambda\left(
1-\theta_{_{1}}u\right)  \right)  /\left(  1-\theta_{_{1}}u\right)  dx\\
&  \qquad-\int\nolimits_{\Omega}\operatorname{div}\left(  A\left(  x\right)
\nabla\theta\right)  \varphi dx\\
&  ~~=\int\nolimits_{\Omega}\varphi\theta\left(  \alpha-\lambda\left(
1-\theta_{_{1}}u\right)  \right)  /\left(  1-\theta_{_{1}}u\right)
+\left\langle A\left(  x\right)  \nabla\theta,\nabla\varphi\right\rangle dx\\
&  ~~\equiv p_{1}\left(  \theta,\varphi\right) .
\end{align*}
If either of the two conditions of the proposition is satisfied, we may use
the Lax-Milgram theorem to obtain the desired result.
\end{proof}

\begin{corollary}
The principal spectrum of $-\pounds $ is contained in $D_{0}^{\ast}%
\equiv\left\{  \lambda\in%
%TCIMACRO{\U{2102} }%
%BeginExpansion
\mathbb{C}
%EndExpansion
^{\ast};\operatorname{Re}\left(  \lambda\right)  \leq0\right\}  $.
\end{corollary}

\begin{proof}
From assumption $\left(  H4\right)  $, $\pounds $ is maximal monotone and
$S_{\pounds }$ is a contraction semigroup. Since $S_{\pounds }$ is a
contraction semigroup, the resolvant set $\rho\left(  -\pounds \right)  $ of
$-\pounds $ contains $%
%TCIMACRO{\U{211d} }%
%BeginExpansion
\mathbb{R}
%EndExpansion
_{+}\backepsilon\left\{  0\right\}  $\footnote{See Theorem 3.1 in \cite{pazy},
p8.} and $\left\Vert S_{\pounds }\left(  t\right)  \right\Vert \leq1,$
$\forall t\in\left[  0,T\right]  $. Therefore, the spectral radius of
$S_{\pounds }\left(  t\right)  $ is less than one. On the other hand
$0\notin\exp\left(  t\sigma_{p}\left(  -\pounds \right)  \right)
\subseteq\sigma_{p}\left(  S_{\pounds }\left(  t\right)  \right)
\subseteq\left\{  0\right\}  \bigcup\exp\left(  t\sigma_{p}\left(
-\pounds \right)  \right)  ,$ $\forall t\in\left[  0,T\right]  $. Clearly, if
$\lambda=\operatorname{Re}\left(  \lambda\right)  +i\operatorname{Im}\left(
\lambda\right)  $ is an element of the principal spectrum of $-\pounds $ then
$\exp\left(  \lambda t\right)  $ is an element of the principal spectrum of
$S_{\pounds }\left(  t\right)  $ and $\left\vert \exp\left(  \lambda t\right)
\right\vert =\exp\left(  \operatorname{Re}\left(  \lambda\right)  t\right)
\left\Vert S_{\pounds }\left(  t\right)  \right\Vert \leq1$. It follows that
$\operatorname{Re}\left(  \lambda\right)  \leq0$.
\end{proof}

\begin{corollary}
The equilibrium $\theta^{\ast}$ is stable. Moreover if all the complex
eigenvalues $\lambda$ of the operator $\theta\mapsto\operatorname{div}\left(
A\left(  x\right)  \nabla\theta\right)  $ satisfy $\alpha\geq\left(
1-\theta_{_{1}}u\right)  \operatorname{Re}\left(  \lambda\right)  $ almost
everywhere in $\Omega$, then $\theta^{\ast}$ asymptotically stable.
\end{corollary}

\subsection{Optimal control of the diffusion model\label{DiffusionModel3}}

\qquad In the previous section we have seen that the equilibrium of system
$\left(  \ref{ModelIntraHoteDiff12}\right)  -\left(
\ref{ModelIntraHoteDiff32}\right)  $ was conditionally asymptotically stable.
Whether or not the equilibrium $\theta^{\ast}$ is asymptotically stable, the
disease progression shall be contained with respect to some criteria given in
terms of costs. The aim of this section is to control the system such that the
following cost functional is minimized:
\[
J_{T}^{3}\left(  u\right)  =\int\nolimits_{0}^{T}\int\nolimits_{\Omega}\left(
\theta^{2}+k_{1}\left(  x\right)  u^{2}\right)  dxdt+\int\nolimits_{\Omega
}k_{2}\left(  x\right)  \theta^{2}\left(  T,x\right)  dx,
\]
where $k_{1}>0,k_{2}\geq0$ are bounded penalization terms. The function
$k_{1}(x)$ can be interpreted as the cost ratio related to the use of control
effort $u$; while $k_{2}$ is the cost ratio related to the magnitude of the
final inhibition rate $\theta\left(  T,\cdot\right)  $. In practice, $k_{1}$
reflects the spatial dependence of environmental sensitivity to control means,
while $k_{2}$ reflects geographical variations in the cost of the inhibition
rate of \textit{Colletotrichum} at the end of the control period.

In order to establish the optimal control, we will first need to define
$U^{K,C}$ as the set of controls $u\in C\left(  \left[  0,T\right]
;H^{1}\left(  \Omega;\left[  0,1\right]  \right)  \right) $ such that for
every $t,s\in\left[  0,T\right]  ,$ $\left\Vert u\left(  t,\cdot\right)
-u\left(  s,\cdot\right)  \right\Vert _{H^{1}\left(  \Omega\right)  }\leq
K\left\vert t-s\right\vert $ and $\left\Vert \nabla u\left(  t,\cdot\right)
\right\Vert _{L^{2}}\leq C.$ For every $K,C\geq0$, $U^{K,C}$ is nonempty.

\begin{theorem}
Let $K,C\geq0$. Then there is a control $v\in U^{K,C}$ which minimizes the
cost $J_{T}^{3}$.
\end{theorem}

\begin{proof}
Since $J_{T}^{3}$ is greater than zero it is bounded below. Let that infinimum
be $J^{\ast}$. There is a sequence $\left(  u_{n}\right)  _{n\in%
%TCIMACRO{\U{2115} }%
%BeginExpansion
\mathbb{N}
%EndExpansion
}$ such that the sequence $\left(  J_{T}^{3}\left(  u_{n}\right)  \right)
_{n\in%
%TCIMACRO{\U{2115} }%
%BeginExpansion
\mathbb{N}
%EndExpansion
}$ converges to $J^{\ast}$. Using definition of $U^{K,C}$ the $\left(
u_{n}\right)  _{n\in%
%TCIMACRO{\U{2115} }%
%BeginExpansion
\mathbb{N}
%EndExpansion
}$ is bounded and uniformly equicontinuous on $\left[  0,T\right]  $. By the
Ascoli theorem, there is a subsequence $\left(  u_{n_{k}}\right)  $ which
converges to a control $v$. Since the cost function is continuous with respect
to $u$ it follows that $J_{T}^{3}\left(  v\right)  =J^{\ast}.$
\end{proof}

We first look the linearized system in the neighborhood of $\left(
\theta,u\right)  =\left(  \varepsilon,0\right)  $, where $\varepsilon$ depends
on $x$. Indeed, this case is of practical significance since the monitoring is
assumed to be continuous year-round, and the endemic period corresponds to
particular conditions. Thus the outbreak of the disease is "observable" at the
moment of onset. The linearized version of $\left(  \ref{ModelIntraHoteDiff12}%
\right)  $\ is
\begin{equation}
\partial\theta/\partial t=\alpha-\alpha\theta-\alpha\varepsilon u\theta_{_{1}%
}+\operatorname{div}\left(  A\left(  x\right)  \nabla\theta\right)  ,\text{ on
}\left]  0,T\right[  \times\Omega\label{ModelIntraHoteDiff12Lin}%
\end{equation}
Note that if $\varepsilon=0$ the linearized system is not controllable.

Let $\pounds _{1}\theta=-\alpha\theta+\operatorname{div}\left(  A\left(
x\right)  \nabla\theta\right)  $. Equation $\left(
\ref{ModelIntraHoteDiff12Lin}\right)  $ becomes
\[
\partial\theta/\partial t=\pounds _{1}\theta-\alpha\varepsilon u\theta_{_{1}%
}+\alpha,\text{ on }\left]  0,T\right[  \times\Omega
\]

\begin{theorem}
The linearized version of $\left(  \ref{ModelIntraHoteDiff12}\right)  -\left(
\ref{ModelIntraHoteDiff32}\right)  $ has an optimal control in $C\left(
\left[  0,T\right]  ;L^{2}\left(  \Omega\right)  \right)  $ given by
\[
u\left(  t,\cdot\right)  =\left(  1/k_{1}\right)  BP\left(  T-t,\cdot\right)
\theta\left(  t,\cdot\right)  +1/(\varepsilon\theta_{_{1}}),\text{ }%
t\in\left[  0,T\right]
\]
where the linear operator $P$ is solution to the following Riccati equation:
\[
\overset{\cdot}{P}=\pounds _{1}P+P\pounds _{1}-\left(  1/k_{1}\right)
PB^{2}P+I,\text{ }P\left(  0\right)  =k_{2}I.
\]
In that equation $I$ is the identity linear operator and $B$ is the linear
operator $\alpha\varepsilon\theta_{_{1}}I$.
\end{theorem}

\begin{proof}
(Sketch)

If we set $v=u-1/\left(  \varepsilon\theta_{_{1}}\right)  $ then equation
$\left(  \ref{ModelIntraHoteDiff12Lin}\right)  $ becomes
\[
\partial\theta/\partial t=\pounds _{1}\theta-\alpha\varepsilon v\theta_{_{1}%
},\text{ on }\left]  0,T\right[  \times\Omega.
\]
The rest of the proof is similar to the proof in \cite{jerzy} concerning
linear regulators.
\end{proof}

If $S_{\pounds _{1}}$ is the contraction semigroup generated by $\pounds _{1}
$, then we have $\forall t\in\left[  0,T\right]  $%
\[
P\left(  t\right)  f=S_{\pounds _{1}}\left(  t\right)  P\left(  0\right)
S_{\pounds _{1}}f
\]%
\[
+\int\nolimits_{0}^{t}S_{\pounds _{1}}\left(  t-s\right)  \left(  I-\left(
1/k_{1}\right)  PB^{2}P\right)  S_{\pounds _{1}}\left(  t-s\right)  fds
\]
Let now consider the nonlinear equation $\left(  \ref{ModelIntraHoteDiff12}%
\right)  $. Let $\pounds _{u}$ be the operator $\pounds $ corresponding to
control strategy $u$ and let $S_{\pounds _{u}}$ be the contraction semigroup
generated by $-\pounds _{u}$. Let
\[
U\equiv\left\{
\begin{array}
[c]{c}%
u\in C\left(  \left[  0,T\right]  ;H^{1}\left(  \Omega;\left[  0,1\right]
\right)  \right)  ;\\
\forall t\in\left[  0,T\right]  ,\text{ }S_{\pounds _{u}}\left(  t\right)
\text{ is invertible}%
\end{array}
\right\}  .
\]
Some necessary and sufficient conditions for a semigroup of operators to be
embedded in a group of operators are given in \cite{pazy}.

\begin{theorem}
Assume that there is a bounded admissible control $u^{\ast}\in U$ which
minimizes the cost function $J_{T}^{3}$. Let $\widetilde{\theta}$ be the
absolutely continuous solution of $\left(  \ref{ModelIntraHoteDiff12}\right)
-\left(  \ref{ModelIntraHoteDiff32}\right)  $ associated with $u^{\ast}$.
Then
\[
\int\nolimits_{\Omega}\left(  \left(  \widetilde{\theta}\left(  t_{0}%
,x\right)  \right)  ^{2}+k_{1}\left(  x\right)  \left(  u^{\ast}\left(
t_{0},x\right)  \right)  ^{2}-p\left(  t\right)  \pounds _{u^{\ast}}%
\widetilde{\theta}\left(  t,x\right)  \right)  dx
\]%
\[
\leq\int\nolimits_{\Omega}\left(  \left(  \widetilde{\theta}\left(
t_{0},x\right)  \right)  ^{2}+k_{1}\left(  x\right)  \left(  u\left(
t_{0},x\right)  \right)  ^{2}-p\left(  t\right)  \pounds _{u}\widetilde
{\theta}\left(  t,x\right)  \right)  dx,
\]
where $p$ is the absolutely continuous solution on $\left[  0,T\right]  $ of
the adjoint state problem
\begin{equation}
\label{AdjointStateProblem}\left\{
\begin{array}
[c]{l}%
\partial p/\partial t=\pounds _{u^{\ast}}p-2\widetilde{\theta},\text{ }\left(
t,x\right)  \in%
%TCIMACRO{\U{211d} }%
%BeginExpansion
\mathbb{R}
%EndExpansion
_{+}^{\ast}\times\Omega\\
\left\langle A\left(  x\right)  \nabla p,n\right\rangle =0,\text{ on }%
%TCIMACRO{\U{211d} }%
%BeginExpansion
\mathbb{R}
%EndExpansion
_{+}^{\ast}\times\partial\Omega\\
p\left(  T\right)  =2k_{2}\widetilde{\theta}\left(  T,\cdot\right)
\end{array}
\right.
\end{equation}

\end{theorem}

\begin{proof}
We give a proof following the maximum principle proof in \cite{jerzy}.

For an arbitrary control $w$ and sufficiently small $h\geq0$, define the
needle variation of $u^{\ast}$ as
\[
u^{h}\left(  t\right)  =\left\{
\begin{array}
[c]{l}%
u^{\ast}\left(  t\right)  \text{, }t\in\left[  0,t_{0}-h\right] \\
w\text{, }t\in\left]  t_{0}-h,t_{0}\right[ \\
u^{\ast}\left(  t\right)  \text{, }t\in\left[  t_{0},T\right]
\end{array}
\right.
\]
Let $\theta^{h}$ be the output corresponding to $u^{h}$. Since $u^{\ast}$
minimizes $J_{T}^{3}$, $J_{T}^{3}\left(  \theta^{h}\right)  >J_{T}^{3}\left(
\theta^{0}\right)  $ and $\partial^{+}J_{T}^{3}\left(  \theta^{0}\right)
/\partial h>0$.
\begin{align*}
&  \partial^{+}\theta^{0}\left(  t_{0},\cdot\right)  /\partial h\\
&  =\underset{h\rightarrow0^{+}}{\lim}\frac{1}{h}\left[  \theta^{h}\left(
t_{0},\cdot\right)  -\theta^{0}\left(  t_{0},\cdot\right)  \right] \\
&  =\underset{h\rightarrow0^{+}}{\lim}\frac{1}{h}\int\nolimits_{t_{0}%
-h}^{t_{0}}\left(  \pounds _{u^{0}}\widetilde{\theta}\left(  s,\cdot\right)
-\pounds _{u^{h}}\theta^{h}\left(  s,\cdot\right)  \right)  ds\\
&  =\left(  \pounds _{u^{0}}-\pounds _{w}\right)  \theta^{0}\left(
t_{0},\cdot\right)
\end{align*}
Since $v^{h}\left(  t\right)  =u^{\ast}\left(  t\right)  $ on $\left[
t_{0},T\right]  $, for almost $t$ in $\left[  t_{0},T\right]  $,
$\partial\theta^{h}/\partial t=\alpha-\pounds _{u^{\ast}}\theta^{h}$.
\begin{align*}
&  \partial\left(  \partial^{+}\theta^{0}\left(  t,\cdot\right)  /\partial
h\right)  /\partial t\\
&  =\left.  \partial\left(  \partial^{+}\theta^{h}\left(  t,\cdot\right)
/\partial h\right)  /\partial t\right\vert _{h=0}\\
&  =\left.  \partial^{+}\left(  \partial\theta^{h}\left(  t,\cdot\right)
/\partial t\right)  /\partial h\right\vert _{h=0}\\
&  =\left.  \partial^{+}\left(  \alpha-\pounds _{u^{\ast}}\theta^{h}\right)
/\partial h\right\vert _{h=0}\\
&  =\left.  -\pounds _{u^{\ast}}\left(  \partial^{+}\theta^{h}/\partial
h\right)  \right\vert _{h=0}\\
&  =-\pounds _{u^{\ast}}\left(  \partial^{+}\theta^{0}/\partial h\right)
\end{align*}
Therefore,
\[
\partial^{+}\theta^{0}\left(  t,\cdot\right)  /\partial h=S_{\pounds _{u^{\ast
}}}\left(  t\right)  \left(  S_{\pounds _{u^{\ast}}}\left(  t_{0}\right)
\right)  ^{-1}\left(  \pounds _{u^{0}}-\pounds _{w}\right)  \theta^{0}\left(
t_{0},\cdot\right)
\]
Consequently,
\begin{align*}
&  \partial^{+}\left(  \int\nolimits_{\Omega}k_{2}\left(  x\right)  \left(
\theta^{0}\left(  T,x\right)  \right)  ^{2}dx\right)  /\partial h\\
&  =2\int\nolimits_{\Omega}k_{2}\left(  x\right)  \theta^{0}\left(
T,x\right)  \partial^{+}\theta^{0}\left(  T,x\right)  /\partial hdx\\
&  =2\int\nolimits_{\Omega}k_{2}\left(  x\right)  \theta^{0}\left(
T,x\right)  S_{\pounds _{u^{\ast}}}\left(  T\right)  \left(
S_{\pounds _{u^{\ast}}}\left(  t_{0}\right)  \right)  ^{-1}\\
&  \qquad\qquad\times\left(  \pounds _{u^{0}}-\pounds _{w}\right)  \theta
^{0}\left(  t_{0},x\right)  dx\\
&  =2\int\nolimits_{\Omega}\left(  S_{\pounds _{u^{\ast}}}\left(
t_{0}\right)  \right)  ^{-1}S_{\pounds _{u^{\ast}}}\left(  T\right)
k_{2}\left(  x\right)  \theta^{0}\left(  T,x\right) \\
&  \qquad\qquad\times\left(  \pounds _{u^{0}}-\pounds _{w}\right)  \theta
^{0}\left(  t_{0},x\right)  dx,
\end{align*}
and in the same manner
\begin{align*}
&  \partial^{+}\left(  \int\nolimits_{0}^{T}\int\nolimits_{\Omega}\left(
\theta^{0}\left(  t,x\right)  \right)  ^{2}dxdt\right)  /\partial h\\
&  =\int\nolimits_{t_{0}}^{T}\int\nolimits_{\Omega}\left(  \partial^{+}\left(
\theta^{0}\left(  t,x\right)  \right)  ^{2}/\partial h\right)  dxdt\\
&  =2\int\nolimits_{t_{0}}^{T}\int\nolimits_{\Omega}\left(
S_{\pounds _{u^{\ast}}}\left(  t_{0}\right)  \right)  ^{-1}%
S_{\pounds _{u^{\ast}}}\left(  t\right)  \theta^{0}\left(  t,x\right) \\
&  \qquad\qquad\quad\times\left(  \pounds _{u^{0}}-\pounds _{w}\right)
\theta^{0}\left(  t_{0},x\right)  dxdt.
\end{align*}
Since $\partial^{+}J_{T}^{3}\left(  \theta^{0}\right)  /\partial h\geq0$, we
have
\begin{align*}
&  ~~2\int\nolimits_{t_{0}}^{T}\int\nolimits_{\Omega}\left(
S_{\pounds _{u^{\ast}}}\left(  t_{0}\right)  \right)  ^{-1}%
S_{\pounds _{u^{\ast}}}\left(  t\right)  \theta^{0}\left(  t,x\right)
\pounds _{u^{0}}\theta^{0}\left(  t_{0},x\right)  dxdt\\
&  ~~+2\int\nolimits_{\Omega}\left(  S_{\pounds _{u^{\ast}}}\left(
t_{0}\right)  \right)  ^{-1}S_{\pounds _{u^{\ast}}}\left(  T\right)
k_{2}\left(  x\right)  \theta^{0}\left(  T,x\right)  \pounds _{u^{0}}%
\theta^{0}\left(  t_{0},x\right)  dx\\
&  ~~-\int\nolimits_{\Omega}\left(  \left(  \widetilde{\theta}\left(
t_{0},x\right)  \right)  ^{2}+k_{1}\left(  x\right)  \left(  u^{\ast}\left(
t_{0},x\right)  \right)  ^{2}\right)  dx\\
&  \geq2\int\nolimits_{t_{0}}^{T}\int\nolimits_{\Omega}\left(
S_{\pounds _{u^{\ast}}}\left(  t_{0}\right)  \right)  ^{-1}%
S_{\pounds _{u^{\ast}}}\left(  t\right)  \theta^{0}\left(  t,x\right)
\pounds _{w}\theta^{0}\left(  t_{0},x\right)  dxdt\\
&  ~~+2\int\nolimits_{\Omega}\left(  S_{\pounds _{u^{\ast}}}\left(
t_{0}\right)  \right)  ^{-1}S_{\pounds _{u^{\ast}}}\left(  T\right)
k_{2}\left(  x\right)  \theta^{0}\left(  T,x\right)  \pounds _{w}\theta
^{0}\left(  t_{0},x\right)  dx\\
&  ~~-\int\nolimits_{\Omega}\left(  \left(  \widetilde{\theta}\left(
t_{0},x\right)  \right)  ^{2}+k_{1}\left(  x\right)  \left(  w\left(
t_{0},x\right)  \right)  ^{2}\right)  dx.
\end{align*}
Note that $t_{0}$ has been chosen arbitrarily. Let $p$ be the solution of the
adjoint state problem (\ref{AdjointStateProblem}). Then for every time
$t\in\left[  0,T\right]  $,%
\[
\int\nolimits_{\Omega}\left(  \left(  \widetilde{\theta}\left(  t_{0}%
,x\right)  \right)  ^{2}+k_{1}\left(  x\right)  \left(  u^{\ast}\left(
t_{0},x\right)  \right)  ^{2}-p\left(  t\right)  \pounds _{u^{\ast}}%
\widetilde{\theta}\left(  t,x\right)  \right)  dx
\]%
\[
\leq\int\nolimits_{\Omega}\left(  \left(  \widetilde{\theta}\left(
t_{0},x\right)  \right)  ^{2}+k_{1}\left(  x\right)  \left(  u\left(
t_{0},x\right)  \right)  ^{2}-p\left(  t\right)  \pounds _{u}\widetilde
{\theta}\left(  t,x\right)  \right)  dx.
\]

\end{proof}

This theorem shows that the optimal control $u^{\ast}$ minimizes the following
Hamiltonian.%
\[
H\left(  \widetilde{\theta},p,u\right)  =\int\nolimits_{\Omega}\left(
\widetilde{\theta}^{2}+k_{1}u^{2}-p\pounds _{u}\widetilde{\theta}\right)  dx
\]
As a result, we have the following necessary condition corresponding to
$\partial H/\partial u\left(  \widetilde{\theta},p,u^{\ast}\right)  =0$:
\begin{equation}
\int\nolimits_{\Omega}\left(  2k_{1}u^{\ast}-\alpha\theta_{1}\widetilde
{\theta}p/(1-\theta_{1}u^{\ast})^{2}\right)  dx=0.\label{TranversalCondition}%
\end{equation}
Condition $\left(  \ref{TranversalCondition}\right)  $ is satisfied in
particular if
\begin{equation}
2k_{1}u^{\ast}(1-\theta_{1}u^{\ast})^{2}=\alpha\theta_{1}\widetilde{\theta
}p,\label{SConditionForTransversalCondition}%
\end{equation}
which is analogous to (\ref{partialH}) for the within-host model . Then we can
adopt the following corresponding strategy
\[
u^{\ast}\left(  t\right)  =%
\begin{cases}
1 & \text{ when }27\alpha\theta_{_{1}}^{2}\theta p\geq8k_{1},\\
w_{3}\left(  t\right)  & \text{ when }27\alpha\theta_{_{1}}^{2}\theta
p<8k_{1},
\end{cases}
\]
where $w_{3}\left(  t\right)  $ is the element of $\left[  0,\min\left\{
\frac{1}{3\theta_{_{1}}},1\right\}  \right]  $ which is the nearest to the
smallest nonnegative solution of the equation $\left(
\ref{SConditionForTransversalCondition}\right)  $.

\section{Discussion\label{Discussion}}

In this paper two models of anthracnose control have been surveyed. These
models both have the general form
\[
\partial\theta/\partial t=f\left(  t,\theta,u\right)  +g\left(  t\right)  ,
\]
where $f$ is linear in the state $\theta$ but not necessarily in the control
$u$. As far as the authors know, this type of control system has not been
extensively studied. This may be due to the fact that physical control
problems usually do not take this form. The majority of such problems tend to
use "additive" controls (see \cite{coron,lions} for literature on models). But
in models of population dynamics, "mutiplicative" control are often more realistic.

Our first model characterizes the within-host behaviour of the disease. We
were able to explicitly calculate an optimal control strategy that effectively
reduces the inhibition rate compared to the case where no control is used. In
our second model we take into account the spatial spread of the disease by
adding a diffusion term. That makes the model more interesting but
considerably more difficult to analyze. Moreover, visual evaluation appears
more difficult because in this case the state of the system is a function of
three spatial variables plus time. Although we have provided equations
satisfied by the optimal control (for the linearized system), in this paper we
do not give a practical method for computing the optimal control. It is
possible that adapted gradient methods may be used \cite{anita}: this is a
subject of ongoing research.

Our models seems quite theoretical, but could be used for practical
applications if the needed parameters were provided. Indeed, in the literature
\cite{danneberger,dodd,duthie,mouen07} there are several attempts to estimate
these parameters. The principal advantage of our abstract approach is that it
can be used to set automatic means to control the disease which are able to
adapt themselves with respect to the host plant and to the parameters values.

Obviously our models can be improved. In particular, several results are based
on some conditions of smoothness of parameters, and the control strategy is
also very regular. In practice parameters are at most piecewise continuous,
and some control strategies are discontinuous. For instance, cultural
interventions in the farm are like pulses with respect to a certain calendar.
The application of antifungal chemical treatments are also pulses, and the
effects of these treatments though continuous are of limited duration. We are
currently investigating a more general model that takes into account those irregularities.

\begin{center}
ACKNOWLEDGEMENT
\end{center}

The authors thank Professor Sebastian ANITA for discussion with the first
author on modelling diseases spread with diffusion equations. That have been
possible through a stay at the University Alexandru Ioan Cuza in Iasi
(Romania) in the frame of the Eugen Ionescu scholarship program.

\end{document}